\documentclass[11pt]{amsart}
\usepackage{graphicx}
\usepackage{amsmath, amssymb, amsthm}
\usepackage[colorlinks, linkcolor=red, citecolor=blue, urlcolor=blue]{hyperref}
\usepackage{xcolor}
\newtheorem{theorem}{Theorem}[section]
\newtheorem{proposition}[theorem]{Proposition}
\newtheorem{lemma}[theorem]{Lemma}
\newtheorem{corollary}[theorem]{Corollary}
\newtheorem{definition}[theorem]{Definition}

\theoremstyle{definition}

\newcommand{\C}{\mathbb{C}}

\newcommand{\Z}{\mathbb{Z}}

\newcommand{\cA}{\mathcal{A}}
\newcommand{\cM}{\mathcal{M}}

\newcommand{\fF}{\mathfrak{F}}
\newcommand{\supp}{\operatorname{supp}}
\newcommand{\Fix}{\operatorname{Fix}}

\newcommand{\EE}{\mathbb{E}}

\title[Relative ISR for Thompson's group $F$]{Relative invariant subalgebra rigidity for Thompson's group $F$}

\author{Tattwamasi Amrutam}
\email{tattwmasiamrutam@gmail.com}

\author{Artem Dudko}

\email{adudko@impan.pl}
\address{Institute of Mathematics of the Polish Academy of Sciences, Warsaw, Poland}
\date{\today}

\begin{document}

\begin{abstract}
We prove that Thompson's group $F$ satisfies the relative invariant subalgebra rigidity property with respect to its commutator subgroup: every von Neumann subalgebra of $L(F)$ that is invariant under conjugation by $[F,F]$ is of the form $L(N)$ for some normal subgroup $N \trianglelefteq F$. Along the way, we establish a general factoriality criterion for invariant subalgebras whose hypotheses are met whenever the ambient group is i.c.c., simple, and every faithful ergodic measure-preserving action of it on a probability space is essentially free. 
\end{abstract}

\maketitle

\section{Introduction}
The interplay between geometric group theory and Popa's deformation/rigidity paradigm \cite{Popa07} has driven substantial advancements in the structural classification of von Neumann algebras arising from countable discrete groups. Within this rich area of study, a prominent line of inquiry focuses on understanding the rigidity phenomena associated with von Neumann subalgebras $\cA \le L(\Gamma)$ that are normalized by the full ambient group $\Gamma$. Inspired by the pioneering investigations of Chifan and Das \cite{CD20} into negatively curved groups, as well as Alekseev and Brugger \cite{AB21} on lattices, Kalantar and Panagopoulos \cite{KP23} proved a striking structural result. They showed that for irreducible lattices inside higher-rank semisimple Lie groups, any $\Gamma$-invariant von Neumann subalgebra of $L(\Gamma)$ is trivially generated by a normal subgroup of $\Gamma$.

Building upon these discoveries, the first-named author and Jiang \cite{AJ23} formally defined the \emph{invariant subalgebra rigidity (ISR) property} to systematically explore this phenomenon across broader classes of countable discrete groups. By definition, a group $\Gamma$ exhibits the ISR property when every von Neumann subalgebra of $L(\Gamma)$ that is invariant under the conjugation action of $\Gamma$ is canonically of the form $L(N)$ for a certain normal subgroup $N \trianglelefteq \Gamma$. This conceptual framework has proven highly successful, inspiring a wealth of subsequent generalizations and applications (see, for instance, \cite{CDS23, DJ24, JZ24, ADJS25}; see also \cite{AJ26} for the $C^*$-version). In addition, recent contributions by Jiang, Li, and Liu \cite{JL26a, JL26b} have initiated a thorough analysis of the structural behavior of invariant subalgebras even in environments where the ISR property does not hold (also see \cite{jiang2026factors}). 

A related, yet considerably more delicate, avenue of research examines whether such rigid behavior persists when the subalgebra is only assumed to be invariant under conjugation by a proper, non-trivial normal subgroup $N \trianglelefteq \Gamma$, instead of the entirety of $\Gamma$. This refined notion is captured by the following definition, recently introduced by the first-named author in \cite{Amr26}.
\begin{definition}
A countable discrete group $\Gamma$ is said to have the \emph{relative ISR property} if for every non-trivial normal subgroup $N \trianglelefteq \Gamma$ and every von Neumann subalgebra $\cM \le L(\Gamma)$ invariant under conjugation by $N$, one has $\cM = L(K)$ for some subgroup $K \le \Gamma$.
\end{definition}
In the work \cite{Amr26}, the first-named author established the relative ISR property for torsion-free acylindrically hyperbolic groups possessing a trivial amenable radical, and irreducible lattices within higher-rank semisimple Lie groups (for example, $\mathrm{SL}_d(\Z)$ where $d \ge 3$ is an odd integer). In both of these contexts, a pivotal feature is that the ambient group is $C^*$-simple. Specifically, in these earlier works, proving that the invariant von Neumann subalgebra is a subfactor (i.e., its center is trivial) constituted an essential intermediate step. This deduction relied fundamentally on the $C^*$-simplicity of the underlying group (see~\cite{AHO24}).

However, when turning our attention to Thompson's group $F$---the classic group composed of orientation-preserving piecewise linear homeomorphisms of the unit interval $[0,1]$ with dyadic breakpoints and slopes strictly in $\{2^k : k \in \Z\}$---we encounter a significantly different landscape. It remains a celebrated, long-standing open question whether $F$ is amenable \cite{CFP96}. We do not even know whether $F$ is $C^*$-simple. As a result, the standard strategy of exploiting $C^*$-simplicity to force the invariant subalgebra to be a factor completely breaks down for this group. 

To overcome this barrier, we employ a different method. Instead of algebraic $C^*$-simplicity, our new approach relies heavily on topological dynamics, specifically, the notion of compressible actions. Despite its elusive amenability status, $F$ possesses an exceptionally rich internal algebraic structure: its commutator subgroup $[F,F]$ is an infinite simple group \cite{Brown87} that consists exactly of those homeomorphisms in $F$ acting trivially near the endpoints of $[0,1]$, and the abelianization $F/[F,F]$ is  isomorphic to $\Z^2$. We notice that in \cite{DJ24} different topological and algebraic properties were used to prove the ISR property for the class of \emph{approximately finite} groups.

The main result of this paper is the following.

\begin{theorem}
\label{thm:mainintro}
Let $F$ denote Thompson's group, and let $\cM \le L(F)$ be a von Neumann subalgebra invariant under conjugation by the commutator subgroup $[F,F]$. Then there exists a normal subgroup $N \trianglelefteq F$ such that $\cM = L(N)$.
\end{theorem}
Stated differently, this establishes that Thompson's group $F$ enjoys the relative ISR property since every non-trivial normal subgroup of $F$ contains its commutator subgroup. In particular, we show that $F$ has the ISR property, a property that was not previously known. Observe that $F$ has the \textit{nonfactorizable regular character property}, defined in \cite{DJ24}. However, the results of \cite{DJ24} do not imply the ISR property for $F$, since it is not known whether $F$ is non-amenable.

The architecture of our proof integrates two primary components of differing natures. The first component is a broadly applicable criterion for factoriality of invariant subalgebras. We show that if $\Gamma$ is any countable, infinite conjugacy class (i.c.c.), simple group where every faithful ergodic probability measure-preserving action happens to be essentially free, then any $\Gamma$-invariant von Neumann subalgebra of $L(\Gamma)$ must necessarily be a factor (see Theorem \ref{thm:factoriality}). For the specific case of the commutator subgroup $[F,F]$, this essential freeness requirement is fulfilled by the deep results of the second-named author and Medynets \cite{DM14}, who deduced it by proving that the natural topological action of $[F,F]$ on the open interval $(0,1)$ is compressible. This dynamical property serves as our crucial substitute for $C^*$-simplicity in establishing the subfactor condition.

The second crucial component is an internal Fourier-analytic boundary technique executed within the algebra $L(F)$. By observing how elements of $\cM$ behave when conjugated by sequences in $[F,F]$ whose spatial supports contract toward the boundary points $0$ and $1$, we can express the standard conditional expectation $\EE_{L([F,F])}$ as a limit in the weak operator topology. This procedure directly implies that $\EE_{L([F,F])}(\cM)$ is contained within $\cM$. Consequently, the aforementioned factoriality principle guarantees that the intersection $\cM \cap L([F,F])$ is a subfactor. Because $[F,F]$ is simple, a sharp structural dichotomy emerges. Either $L([F,F])$ is contained in $\cM$, or the intersection $\cM \cap L([F,F])$ is trivial. In the former scenario, an appeal to the well-known Galois correspondence for intermediate subalgebras of the inclusion $L([F,F]) \subseteq L(F)$ swiftly concludes the argument. In the latter scenario, a GNS vanishing technique used in \cite{DJ24} guarantees that $\cM = \C$.
\subsection{Organization.} In Section \ref{sec:prelim} we collect the requisite background on group von Neumann algebras, Thompson's group $F$, and the compressible actions of Dudko--Medynets. Section \ref{sec:factoriality} establishes the general factoriality principle and its consequence for $[F,F]$-invariant subalgebras of $L([F,F])$. Section \ref{sec:main} contains the boundary Fourier coefficient lemma and the proof of the Main Theorem.
\subsection{Acknowledgements.} This work was partially supported by the Simons Foundation grant (award no. SFI-MPS-T-Institutes-00010825) and from State Treasury funds as part of a task commissioned by the Minister of Science and Higher Education under the project “Organization of the Simons Semesters at the Banach Center - New Energies in 2026-2028” (agreement no. MNiSW/2025/DAP/491). 
\vspace{0.5cm} 
\begin{center}
    \includegraphics[width=0.5\textwidth]{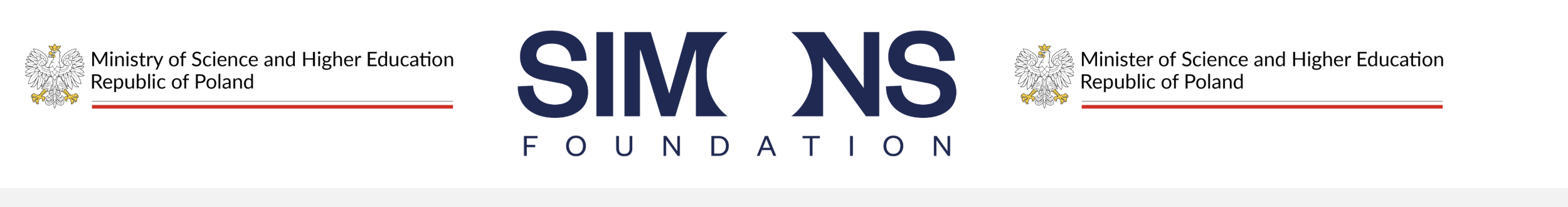} 
\end{center}
The authors thank Yongle Jiang and Adam Skalski for many helpful discussions and suggestions.
\section{Preliminaries}\label{sec:prelim}

\subsection{Group von Neumann algebras}

Let $\Gamma$ be a countable discrete group. We write $\lambda \colon \Gamma \to B(\ell^2 \Gamma)$ for the left regular representation, $\lambda(g) \delta_h = \delta_{gh}$, where $B(\ell^2 \Gamma)$ is the algebra of bounded linear operators on the Hilbert space $\ell^2 \Gamma$. The group von Neumann algebra $L(\Gamma)$ is the weak operator closure of the linear span of $\{\lambda(g) : g \in \Gamma\}$ in $B(\ell^2 \Gamma)$. It carries the canonical faithful normal tracial state $\tau$ given by $\tau(\lambda(g)) = \delta_{g,e}$.

Every element $x \in L(\Gamma)$ admits a Fourier expansion
\[
x = \sum_{g \in \Gamma} x_g \, \lambda(g), \qquad x_g := \tau(x \lambda(g)^*) \in \C,
\]
which converges in $\|\cdot\|_2$, where $\|x\|_2 := \tau(x^* x)^{1/2}$. The function $g \mapsto x_g$ is square-summable, and $x$ is uniquely determined by $\left \{g \in \Gamma : x_g \neq 0\right\}$.

Recall that $\Gamma$ is said to have \emph{infinite conjugacy classes (i.c.c.)} if every conjugacy class except that of the identity is infinite. Equivalently, $L(\Gamma)$ is a $\mathrm{II}_1$ factor; in particular, its center is trivial.

For any subgroup $\Lambda \le \Gamma$, there is a canonical faithful normal conditional expectation $\EE_{L(\Lambda)} \colon L(\Gamma) \to L(\Lambda)$, determined on generators by
\[
\EE_{L(\Lambda)}(\lambda(g)) = \begin{cases} \lambda(g) & \text{if } g \in \Lambda, \\ 0 & \text{otherwise.} \end{cases}
\]
We refer the readers to \cite[Chapter~2]{brown2008textrm} for more details. 
\subsection{Thompson's group \texorpdfstring{$F$}{}}

Thompson's group $F$ is the group of orientation-preserving piecewise linear homeomorphisms $g \colon [0,1] \to [0,1]$ such that
\begin{enumerate}
\item every breakpoint of $g$ lies in the dyadic rationals $\Z[1/2] \cap [0,1]$, and
\item the slope of $g$ at every point of differentiability lies in $\{2^k : k \in \Z\}$.
\end{enumerate}
For an element $g \in F$, we write $g'(0)$ and $g'(1)$ for the one-sided derivatives at $0$ and $1$ respectively; both lie in $\{2^k : k \in \Z\}$. The support of $g$ is denoted $\supp(g) := \{x \in [0,1] : g(x) \neq x\}$.

The basic algebraic structure of $F$, established by Brown \cite{Brown87}, is summarized in the following.

\begin{proposition}\label{prop:F-structure}
Let $F$ be Thompson's group.
\begin{enumerate}
\item $F$ is torsion-free and i.c.c.
\item The commutator subgroup $[F,F]$ is simple, infinite, and consists precisely of those elements $g \in F$ for which $g'(0) = g'(1) = 1$; equivalently, $g$ is the identity on some neighborhood of both $0$ and $1$.
\item The abelianization $F/[F,F]$ is isomorphic to $\Z^2$.
\end{enumerate}
\end{proposition}
In particular, $[F,F]$ is itself i.c.c.; consequently both $L(F)$ and $L([F,F])$ are $\mathrm{II}_1$ factors. We will write $\supp(g) \to 1$ for a sequence $\{g_k\} \subset F$ when $\supp(g_k) \subseteq (1 - \varepsilon_k, 1)$ for some sequence $\varepsilon_k \to 0$; similarly for $\supp(g) \to 0$. In view of Proposition \ref{prop:F-structure}(2), sequences in $[F,F]$ with $\supp(g_k) \to 1$ or $\supp(g_k) \to 0$ are abundant. Indeed, any non-trivial element of $[F,F]$ may be ``transported'' arbitrarily close to either endpoint by conjugation within $F$.

\subsection{Compressible actions and Essential freeness} The following notion was introduced by the second-named author and Medynets in \cite{DM14} as a topological-dynamical mechanism for ruling out non-regular characters on certain transformation groups.
\begin{definition}\cite[Definition 2.5]{DM14}, 
Let $\Gamma$ act by homeomorphisms on a regular Hausdorff topological space $X$. The action is called \emph{compressible} if there exists a basis $\mathcal{U}$ of the topology on $X$ such that
\begin{enumerate}
\item for all $g \in \Gamma$, there exists $U \in \mathcal{U}$ with $\supp(g) \subseteq U$;
\item for all $U_1, U_2 \in \mathcal{U}$, there exists $g \in \Gamma$ with $g(U_1) \subseteq U_2$;
\item for all $U_1, U_2, U_3 \in \mathcal{U}$ with $\overline{U}_1 \cap \overline{U}_2 = \emptyset$, there exists $g \in \Gamma$ with $g(U_1) \cap U_3 = \emptyset$ and $\supp(g) \cap U_2 = \emptyset$;
\item for all $U_1, U_2 \in \mathcal{U}$, there exists $U_3 \in \mathcal{U}$ with $U_3 \supseteq U_1 \cup U_2$.
\end{enumerate}
\end{definition}

Recall that an action of $\Gamma$ on a probability measure space $(X, \mu)$ is \emph{essentially free} if $\mu(\Fix(g)) = 0$ for every $g \in \Gamma \setminus \{e\}$, where $\Fix(g) := \{x \in X : g \cdot x = x\}$.

The key dynamical input from \cite{DM14} that we shall use is the following.

\begin{theorem}\cite[Lemma 3.2, Corollary 3.3, Corollary 3.10]{DM14}\label{thm:DM}
The natural action of $[F,F]$ on $(0,1)$ by piecewise linear homeomorphisms is compressible. Consequently, every faithful ergodic measure-preserving action of $[F,F]$ on a standard probability space is essentially free.
\end{theorem}

\subsection{Suzuki-type partitions of unity}

The technical lemma below is a Rokhlin-type construction for free actions; it appears in essentially this form in \cite{suzuki2020complete} and constituted an important part in his classification of intermediate algebras.

\begin{lemma}\label{lem:suzuki}
Let $\Gamma \curvearrowright (X, \mu)$ be an essentially free probability measure-preserving action, and let $\fF \subset \Gamma \setminus \{e\}$ be a finite subset. Then there exist $m\in\mathbb N$ and non-negative functions $f_1, \dots, f_m \in L^\infty(X, \mu)^+$ such that
\[
\sum_{j=1}^m f_j^2 = 1 \quad \text{and} \quad (s \cdot f_j) \, f_j = 0 \quad \text{for every } s \in \fF \text{ and every } j.
\]
\end{lemma}

\section{Factoriality of invariant subalgebras}\label{sec:factoriality}

The aim of this section is to establish a general factoriality result for invariant subalgebras. Our key technical input is the following proposition, which rules out equivariant embeddings of essentially free dynamical systems into the center of an invariant subalgebra.

\begin{proposition}\label{prop:no-equivariant-hom}
Let $\Gamma$ be a countable i.c.c.\ group and let $\cM \le L(\Gamma)$ be a von Neumann subalgebra with $\cM \neq \C$. Assume that $\Gamma \curvearrowright (X, \mu)$ is an essentially free probability measure-preserving action. Then there does not exist a $\Gamma$-equivariant, state-preserving unital $*$-homomorphism
\[
\theta \colon L^\infty(X, \mu) \to Z(\cM).
\]
\end{proposition}

\begin{proof}
Assume towards a contradiction that such a map $\theta$ exists. We shall show that this forces $\cM = \C$. Let $a \in \cM$; our goal is to demonstrate that $a = \tau(a) 1$, where $\tau$ is the canonical trace on $L(\Gamma)$. Fix $\varepsilon > 0$. Since the linear span of $\{\lambda(g) : g \in \Gamma\}$ is $\|\cdot\|_2$-dense in $L(\Gamma)$, there exist a finite subset $\fF \subset \Gamma \setminus \{e\}$ and complex coefficients $\{c_s\}_{s \in \fF}$ such that
\[
\left\| a - \left( \sum_{s \in \fF} c_s \lambda(s) + \tau(a) 1 \right) \right\|_2 < \varepsilon.
\]
Set $b := \sum_{s \in \fF} c_s \lambda(s) + \tau(a) 1$.
Apply Lemma \ref{lem:suzuki} to the essentially free action and the finite subset $\fF$ to produce non-negative functions $f_1, \dots, f_m \in L^\infty(X, \mu)^+$ with
\[
\sum_{j=1}^m f_j^2 = 1, \qquad (s \cdot f_j) f_j = 0 \quad \text{for all } s \in \fF, \, j = 1, \dots, m.
\]
Define $\psi \colon L(\Gamma) \to L(\Gamma)$ by
\[
\psi(x) := \sum_{j=1}^m \theta(f_j) \, x \, \theta(f_j).
\]
Since $\theta$ is a unital $*$-homomorphism, $\psi$ is completely positive and unital; indeed,
\[
\psi(1) = \sum_{j=1}^m \theta(f_j)^2 = \theta\!\left( \sum_{j=1}^m f_j^2 \right) = \theta(1) = 1.
\]
By the Russo--Dye theorem (see, e.g., \cite[Corollary 2.9]{Paulsen02}), unital positive maps between C*-algebras have norm equal to $1$, so $\|\psi\| = 1$. Moreover, $\psi$ is trace-preserving: for any $x \in L(\Gamma)$, using the trace property,
\[
\tau(\psi(x)) = \sum_{j=1}^m \tau(\theta(f_j) x \theta(f_j)) = \tau\!\left( x \sum_{j=1}^m \theta(f_j)^2 \right) = \tau(x).
\]
Since $\psi$ is unital completely positive, the Kadison--Schwarz inequality $\psi(x)^* \psi(x) \le \psi(x^* x)$ together with trace-preservation gives
\[
\|\psi(x)\|_2^2 = \tau(\psi(x)^* \psi(x)) \le \tau(\psi(x^* x)) = \tau(x^* x) = \|x\|_2^2.
\]
Thus, $\psi$ is a $\|\cdot\|_2$-contraction. Since $a \in \cM$ and $\theta(f_j) \in Z(\cM)$, each $\theta(f_j)$ commutes with $a$. Hence
\[
\psi(a) = \sum_{j=1}^m \theta(f_j) a \theta(f_j) = a \sum_{j=1}^m \theta(f_j)^2 = a.
\]
For $s \in \fF$, the $\Gamma$-equivariance of $\theta$ gives $\lambda(s) \theta(f_j) \lambda(s)^* = \theta(s \cdot f_j)$, hence $\lambda(s) \theta(f_j) = \theta(s \cdot f_j) \lambda(s)$. Therefore
\[
\psi(\lambda(s)) = \sum_{j=1}^m \theta(f_j) \lambda(s) \theta(f_j) = \sum_{j=1}^m \theta(f_j (s \cdot f_j)) \lambda(s) = 0,
\]
since $f_j (s \cdot f_j) = 0$ by construction. By linearity, $\psi(b) = \tau(a) 1$. Combining all of the above and the contraction property of $\psi$,
\[
\|a - \tau(a) 1\|_2 = \|\psi(a) - \psi(b)\|_2 = \|\psi(a - b)\|_2 \le \|a - b\|_2 < \varepsilon.
\]
Since $\varepsilon > 0$ was arbitrary, we conclude $a = \tau(a) 1$. As $a \in \cM$ was arbitrary, $\cM = \C \cdot 1$, contradicting $\cM \neq \C$.
\end{proof}

We can now deduce a clean factoriality statement for invariant subalgebras.

\begin{theorem}\label{thm:factoriality}
Let $\Gamma$ be a countable, i.c.c., simple group with the property that every faithful ergodic measure-preserving action of $\Gamma$ on a standard probability space is essentially free. Then every $\Gamma$-invariant von Neumann subalgebra $\cA \le L(\Gamma)$ is a factor, i.e., $Z(\cA) = \C \cdot 1$.
\end{theorem}

\begin{proof}
Let $\cA \le L(\Gamma)$ be $\Gamma$-invariant. Since $\cA$ is invariant under conjugation by $\Gamma$, its center $Z(\cA)$ is also $\Gamma$-invariant. By the spectral theorem, we may identify $Z(\cA) \cong L^\infty(X, \nu)$ for some standard probability measure space $(X, \nu)$. Let $\Phi \colon L^\infty(X,\nu) \xrightarrow{\sim} Z(\cA)$ denote the 
spectral isomorphism. We equip $(X,\nu)$ with the $\Gamma$-action $\sigma$ 
uniquely determined by 
\[
  \Phi(\sigma_g(f)) \;=\; \lambda(g)\,\Phi(f)\,\lambda(g)^*, 
  \qquad g \in \Gamma,\ f \in L^\infty(X,\nu).
\]
That $\sigma$ acts by a measure-preserving transformation follows from the 
fact that conjugation by $\lambda(g)$ preserves the canonical trace 
restricted to $Z(\cA)$. With this action, the inclusion 
$\theta \colon L^\infty(X,\nu) \hookrightarrow \cA$ obtained by composing 
$\Phi$ with $Z(\cA) \hookrightarrow \cA$ is tautologically 
$\Gamma$-equivariant.

Suppose, towards a contradiction, that $Z(\cA) \neq \C$. We claim first that the action $\Gamma \curvearrowright (X, \nu)$ is ergodic. Indeed, if $f \in Z(\cA)$ is fixed by conjugation by every $\lambda(g)$, $g \in \Gamma$, then $f$ commutes with every element of $L(\Gamma)$, so $f \in L(\Gamma)' \cap L(\Gamma) = Z(L(\Gamma)) = \C$ since $\Gamma$ is i.c.c. Thus, the only $\Gamma$-fixed elements of $Z(\cA)$ are scalars, which is precisely ergodicity.

Since $Z(\cA) \neq \C$, the space $(X, \nu)$ is non-trivial, hence the action of $\Gamma$ on $(X, \nu)$ is non-trivial. The kernel of this action is a normal subgroup of $\Gamma$; by simplicity, it is either $\{e\}$ or $\Gamma$. The latter would force the action to be trivial, contradicting non-triviality (or, equivalently, ergodicity on a non-point space). Hence, the action is faithful. By hypothesis, it is then essentially free.

We now apply Proposition \ref{prop:no-equivariant-hom} with $\cM = \cA$ and $\theta \colon L^\infty(X, \nu) \xrightarrow{\;\sim\;} Z(\cA) \hookrightarrow \cA$ the canonical (identity) inclusion. $\theta$ is equivariant and state-preserving, and the action is essentially free. We need also that $\cA \neq \C$; but if $\cA = \C$ then $Z(\cA) = \C$, contradicting our assumption. Therefore $\cA \neq \C$, and Proposition \ref{prop:no-equivariant-hom} produces the contradiction.

Hence $Z(\cA) = \C$, proving that $\cA$ is a factor.
\end{proof}

\begin{corollary}\label{cor:F-factor}
Let $\cA \le L([F,F])$ be a von Neumann subalgebra invariant under conjugation by $[F,F]$. Then $\cA$ is a factor.
\end{corollary}

\begin{proof}
By Proposition \ref{prop:F-structure}, the group $[F,F]$ is simple, infinite, hence i.c.c. By Theorem \ref{thm:DM}, every faithful ergodic measure-preserving action of $[F,F]$ on a standard probability space is essentially free. The hypotheses of Theorem \ref{thm:factoriality} are therefore satisfied with $\Gamma = [F,F]$, and the conclusion follows.
\end{proof}

\section{Relative ISR for Thompson's group \texorpdfstring{$F$}{}}\label{sec:main}

We now turn to the main theorem. The argument proceeds in two steps: first, a boundary Fourier coefficient lemma which yields that the conditional expectation $\EE_{L([F,F])}$ preserves any $[F,F]$-invariant von Neumann subalgebra of $L(F)$; and second, a dichotomy based on whether the resulting factor $\cM \cap L([F,F])$ is trivial or not.

\subsection{The boundary Fourier coefficient lemma}

\begin{lemma}\label{lem:boundary}
Let $x \in L(F)$, and let $\{g_k\} \subset F$ satisfy $\supp(g_k) \to 1$ in the sense that $\supp(g_k) \neq \emptyset$ and $\supp(g_k) \subset (1-\varepsilon_k, 1)$ for some $\varepsilon_k \to 0$. Then $\lambda(g_k)\, x\, \lambda(g_k)^*$ converges in the weak operator topology to an element $\tilde{x} \in L(F)$ whose Fourier coefficients are
\[
\tilde{x}_h = \begin{cases} x_h & \text{if } h'(1) = 1, \\ 0 & \text{if } h'(1) \neq 1, \end{cases} \qquad h \in F.
\]
The analogous statement holds for $\{f_k\} \subset F$ with $\supp(f_k) \to 0$, with $h'(0)$ in place of $h'(1)$.
\end{lemma}

\begin{proof}
Set $x_k := \lambda(g_k)\, x\, \lambda(g_k)^*$. Each $x_k$ has operator norm $\|x\|$, so the sequence lies in the WOT-compact ball $\overline{B}(0, \|x\|) \subset L(F)$, where WOT stands for the weak operator topology.  Let $\{x_{k_i}\}$ be any WOT-convergent subsequence with limit $\tilde{x} \in L(F)$ (the algebra is WOT-closed). For each $h \in F$, the functional $y \mapsto \tau(y\,\lambda(h)^*) = \langle y\delta_e, \delta_h\rangle$ is WOT-continuous, and the trace property gives
\begin{align*}
\tau\!\left(\lambda(g)\, x\, \lambda(g)^* \lambda(h)^*\right) &= \tau\!\left(x\, \lambda(g)^* \lambda(h)^* \lambda(g)\right) = \tau\!\left(x\, \lambda(g^{-1} h g)^*\right) =x_{g^{-1} h g}.
\end{align*}
Hence
\begin{equation}\label{eq:fourier-limit}
\tilde{x}_h \;=\; \lim_{i \to \infty} x_{\, g_{k_i}^{-1} h\, g_{k_i}}.
\end{equation}
\medskip
\noindent\textbf{Case 1: $h'(1) = 1$.} Every element of $F$ is piecewise affine with finitely many dyadic breakpoints and slopes that are integer powers of $2$. On the rightmost piece, the slope equals $h'(1) = 1$; combined with $h(1) = 1$, this forces $h(x) = x$ on some interval $[1 - \delta, 1]$. Once $\varepsilon_{k_i} < \delta$, the support of $g_{k_i}$ lies in a region where $h$ is the identity, so $h$ and $g_{k_i}$ commute, $g_{k_i}^{-1} h g_{k_i} = h$, and equation~\eqref{eq:fourier-limit} gives $\tilde{x}_h = x_h$.

\medskip
\noindent\textbf{Case 2: $h'(1) \neq 1$.} Write $h'(1) = 2^m$ with $m \in \mathbb{Z} \setminus \{0\}$. Choose $\delta > 0$ small enough so that $h$ is affine on $[1 - \delta, 1]$:
\[
h(x) \;=\; 1 - 2^m(1 - x), \qquad x \in [1 - \delta, 1].
\]
On $[1 - \delta, 1)$ the equation $h(x) = x$ reduces to $(1 - x)(1 - 2^m) = 0$, which forces $x = 1$. Thus $h$ has no fixed points in $[1 - \delta, 1)$.
We now claim that if $f \in F$ commutes with $h$ and $\supp(f) \subset (1 - \delta, 1)$, then $f = e$.
Indeed, $fh = hf$ implies $h(\supp(f)) = \supp(f)$. If $f \neq e$, let $a := \inf \supp(f) \in [1 - \delta, 1)$. Since $h$ is an orientation-preserving homeomorphism of $[0,1]$ and preserves $\supp(f)$ as a set, $h(a) = a$, contradicting the absence of fixed points of $h$ in $[1 - \delta, 1)$.

Further, since $\supp(g_{k_i})\to 1$, passing to a subsequence if necessary, we may assume that $g_{k_i}$ are pairwise distinct. Let $s_i := g_{k_i}^{-1}\, h\, g_{k_i}$. For any $i, j$ with $\varepsilon_{k_i}, \varepsilon_{k_j} < \delta$, suppose $s_i = s_j$. Then $f := g_{k_j}\, g_{k_i}^{-1}$ commutes with $h$ and $\supp(f) \subset (1 - \delta, 1)$. Since $fh = hf$, $h$ preserves $\supp(f)$ as a set; if $f \neq e$, then $\inf \supp(f) \in [1 - \delta, 1)$ is a fixed point of $h$, contradicting the choice of $\delta$. Hence $f = e$, which contradicts $g_{k_i}$ being pairwise distinct. Therefore, $s_i$ are pairwise distinct.
Since $x \in L(F) \hookrightarrow \ell^2(F)$, the family $(x_g)_{g \in F}$ is square-summable, hence $x_{s_i} \to 0$. By equation~\eqref{eq:fourier-limit}, $\tilde{x}_h = 0$.

The case $\supp(f_k) \to 0$ is identical with $0$ in place of $1$: choose $\delta$ such that $h$ is affine on $[0, \delta]$, the equation $h(x) = x$ has no solutions in $(0, \delta]$, and the centralizer fact runs verbatim.
\end{proof}

\begin{proof}[Proof of Theorem~\ref{thm:mainintro}]
We first show that $\EE_{L([F,F])}$ preserves $\cM$. Let $x \in \cM$. Since $\cM$ is $[F,F]$-invariant, $\lambda(g) x \lambda(g)^* \in \cM$ for every $g \in [F,F]$. Choose any sequence $\{g_k\} \subset [F,F]$ with $\supp(g_k) \to 1$ (such sequences exist by Proposition \ref{prop:F-structure}(2)). By Lemma \ref{lem:boundary}, a subsequence of $\{\lambda(g_k) x \lambda(g_k)^*\}$ converges WOT to an element $\tilde{x} \in L(F)$ whose Fourier coefficients vanish outside $\{h \in F : h'(1) = 1\}$. Since $\cM$ is WOT-closed and contains the subsequence, $\tilde{x} \in \cM$.

Now choose a sequence $\{f_k\} \subset [F,F]$ with $\supp(f_k) \to 0$ and apply Lemma \ref{lem:boundary} to $\tilde{x}$. A subsequence of $\{\lambda(f_k) \tilde{x} \lambda(f_k)^*\}$ converges WOT to an element $\tilde{\tilde{x}} \in \cM$ whose Fourier coefficients vanish outside $\{h \in F : h'(0) = 1 \text{ and } h'(1) = 1\}$. By Proposition \ref{prop:F-structure}(2), this last set is exactly $[F,F]$. The element $\tilde{\tilde{x}}$ is therefore the unique element of $L([F,F])$ with the same Fourier coefficients as $x$ on $[F,F]$; that is, $\tilde{\tilde{x}} = \EE_{L([F,F])}(x)$. We conclude that $\EE_{L([F,F])}(x) \in \cM$ for all $x \in \cM$, hence
\[
\EE_{L([F,F])}(\cM) \subseteq \cM \cap L([F,F]).
\]
The reverse inclusion is immediate, so $\EE_{L([F,F])}(\cM) = \cM \cap L([F,F])$. 

We now leverage the character-theoretic input of Dudko--Jiang \cite{DJ24}. 
First, recall that by \cite{DM14}, $[F,F]$ has only two indecomposable characters: the trivial one and the regular one. In particular, $[F,F]$ has the non-factorizable regular character property (see \cite[Proposition~3.10(1)]{DJ24}).
Moreover, $[F,F]$ is i.c.c.\ by Proposition~\ref{prop:F-structure}. The intersection $\cM \cap L([F,F])$ is a $[F,F]$-invariant von Neumann subalgebra of $L([F,F])$, and by Corollary~\ref{cor:F-factor} it is a factor. Applying \cite[Theorem~3.3]{DJ24} to the i.c.c.\ group $[F,F]$ together with this $[F,F]$-invariant subfactor of $L([F,F])$, we conclude that $\cM \cap L([F,F]) = L(H)$ for some normal subgroup $H \trianglelefteq [F,F]$. Since $[F,F]$ is simple (Proposition~\ref{prop:F-structure}), $H = \{e\}$ or $H = [F,F]$, and therefore
\[
\cM \cap L([F,F]) = L([F,F]) \quad \text{or} \quad \cM \cap L([F,F]) = \C.
\]
Suppose first that $\cM \cap L([F,F]) = L([F,F])$. Then $L([F,F]) \subseteq \cM \subseteq L(F)$. Since $[F,F] \trianglelefteq F$ and the abelianization $F/[F,F] \cong \Z^2$ is abelian, the standard Galois correspondence for intermediate von Neumann subalgebras between $L([F,F])$ and $L(F)$ (see \cite[Proposition~4.4]{bedos2023c}, also see \cite[Lemma~3.3]{jiang2021maximal}) asserts that every such intermediate subalgebra is of the form $L(N)$ for some subgroup $[F,F] \le N \le F$. 

Assume now that $\cM \cap L([F,F]) = \C$. We show that this forces $\cM = \C$, which corresponds to the trivial normal subgroup $N = \{e\}$. For every $x \in \cM$, $$\EE_{L([F,F])}(x) \in \EE_{L([F,F])}(\cM)=\cM \cap L([F,F]) = \C.$$ Since $\EE_{L([F,F])}$ is trace-preserving, $\EE_{L([F,F])}(x) = \tau(x) 1$. Hence, for every $n \in [F,F] \setminus \{e\}$ and every $x \in \cM$,
\[
\tau(x \lambda(n)^*) = \tau(\EE_{L([F,F])}(x \lambda(n)^*)) = \tau(\EE_{L([F,F])}(x) \lambda(n)^*) = \tau(x) \tau(\lambda(n)^*) = 0.
\]
This, in particular, implies that
\[
\EE_{\cM}(\lambda(n)) = 0 \quad \text{for all } n \in [F,F] \setminus \{e\}.
\]
We now invoke \cite[Lemma 3.6]{Amr26} with $K := [F,F] \setminus \{e\}$ and $N=[F,F]$. The vanishing condition $\EE_{\cM}(\lambda(s)) = 0$ for $s \in K$ was just established. It remains to verify the thickness of $K$ as defined in \cite[Definition 3.5]{Amr26}. Namely, for any $g \in F \setminus \{e\}$ we need to find an infinite sequence $\{n_i\}_{i \ge 1} \subset [F,F]$ such that $w_{i,j} \in [F,F] \setminus \{e\}$ for all $i \neq j$, where \[
w_{i,j} = (n_j^{-1} g n_j) (n_i^{-1} g^{-1} n_i)
.\] Since $[F,F] \trianglelefteq F$, for any $n_i, n_j \in [F,F]$, the element $w_{i,j}$
lies in the coset $g \cdot g^{-1} \cdot [F,F] = [F,F]$ (computing modulo $[F,F]$). It is sufficient, therefore, to arrange $w_{i,j} \neq e$. Suppose $w_{i,j} = e$. Then $n_j^{-1} g n_j = n_i^{-1} g n_i$, which is equivalent to $n_i n_j^{-1} \in C_F(g)$, where $C_F(g)$ denotes the centralizer of $g$ in $F$. Since $n_i, n_j \in [F,F]$, we have $n_i n_j^{-1} \in C_{[F,F]}(g) := [F,F] \cap C_F(g)$. Thus it suffices to choose $\{n_i\}_{i \ge 1}$ representing distinct right cosets of $C_{[F,F]}(g)$ in $[F,F]$. This is possible whenever $C_{[F,F]}(g)$ has infinite index in $[F,F]$, which we now verify.

Suppose for contradiction that $C_{[F,F]}(g)$ has finite index $m$ in $[F,F]$. The action of $[F,F]$ by left multiplication on the finite coset space $[F,F]/C_{[F,F]}(g)$ yields a group homomorphism $\rho \colon [F,F] \to S_m$. Its kernel is a normal subgroup of $[F,F]$, which is simple by Proposition~\ref{prop:F-structure}, so $\ker \rho = \{e\}$ or $\ker \rho = [F,F]$. The former is impossible, since $[F,F]$ is infinite and cannot embed into the finite group $S_m$. Hence $\ker \rho = [F,F]$, so $\rho$ is trivial; as the left-multiplication action on $[F,F]/C_{[F,F]}(g)$ is transitive, this forces $m = 1$, i.e.\ $C_{[F,F]}(g) = [F,F]$. But then $g$ commutes with every element of $[F,F]$, so $g \in C_F([F,F])$. Since $[F,F]$ acts on $(0,1)$ without a common fixed point and contains elements supported arbitrarily close to each endpoint, its centralizer in $F$ is trivial: $C_F([F,F]) = \{e\}$. This contradicts $g \neq e$. Hence $C_{[F,F]}(g)$ has infinite index in $[F,F]$.

We may therefore select an infinite sequence $\{n_i\}_{i \ge 1} \subset [F,F]$ from pairwise distinct right cosets of $C_{[F,F]}(g)$. For this sequence, $n_i n_j^{-1} \notin C_{[F,F]}(g)$ for $i \neq j$, so $w_{i,j} \neq e$. As argued above, $w_{i,j} \in [F,F] = K \cup \{e\}$, so $w_{i,j} \in K$ for all $i \neq j$. By \cite[Lemma 3.6]{Amr26}, $\cM = \C $, which is of the required form. This finishes the proof.
\end{proof}
\bibliographystyle{amsalpha}
\bibliography{biblio}

@article{AB21,
  author = {Alekseev, V. and Brugger, R.},
  title = {A rigidity result for normalized subfactors},
  journal = {Journal of Operator Theory},
  volume = {86},
  number = {1},
  year = {2021},
  pages = {3--15}
}

@article{ADJS25,
  author = {Amrutam, T. and Dudko, A. and Jiang, Y. and Skalski, A.},
  title = {Invariant subalgebras rigidity for von {N}eumann algebras of groups arising as certain semidirect products},
  journal = {arXiv preprint arXiv:2507.12824},
  year = {2025},
  pages={32}
}

@article{jiang2021maximal,
  title={Maximal von Neumann subalgebras arising from maximal subgroups},
  author={Jiang, Y.},
  journal={Science China Mathematics},
  volume={64},
  number={10},
  pages={2295--2312},
  year={2021},
  publisher={Springer}
}

@article{AJ23,
  author = {Amrutam, T. and Jiang, Y.},
  title = {On invariant von {N}eumann subalgebras rigidity property},
  journal = {Journal of Functional Analysis},
  volume = {284},
  number = {5},
  year = {2023},
  pages = {109804}
}

@article{AJ26,
  author = {Amrutam, T. and Jiang, Y.},
  title = {Invariant {C}{*}-subalgebras of the reduced group {C}{*}-algebra},
  journal = {Journal of Functional Analysis},
  volume = {291},
  number = {1},
  year = {2026},
  pages = {111471}
}

@article{Amr26,
  author = {Amrutam, T.},
  title = {On relative invariant subalgebra rigidity property},
  journal = {arXiv preprint},
  year = {2026},
  pages={24}
}

@article{Brown87,
  author = {Brown, K.},
  title = {Finiteness properties of groups},
  journal = {Journal of Pure and Applied Algebra},
  volume = {44},
  number = {1--3},
  year = {1987},
  pages = {45--75}
}

@article{CFP96,
  author = {Cannon, J.W. and Floyd, W.J. and Parry, W.R.},
  title = {Introductory notes on {R}ichard {T}hompson's groups},
  journal = {Enseign. Math.},
  volume = {42},
  number = {3--4},
  year = {1996},
  pages = {215--256}
}

@article{CD20,
  author = {Chifan, I. and Das, S.},
  title = {Rigidity results for von {N}eumann algebras arising from mixing extensions of profinite actions of groups on probability spaces},
  journal = {Mathematische Annalen},
  volume = {378},
  number = {3},
  year = {2020},
  pages = {907--950}
}

@article{CDS23,
  author = {Chifan, I. and Das, S. and Sun, B.},
  title = {Invariant subalgebras of von {N}eumann algebras arising from negatively curved groups},
  journal = {Journal of Functional Analysis},
  volume = {285},
  number = {9},
  year = {2023},
  pages = {110098}
}

@article{DJ24,
  author = {Dudko, A. and Jiang, Y.},
  title = {A character approach to the {ISR} property},
  journal = {arXiv preprint arXiv:2410.14517},
  year = {2024},
  pages={40}
}

@article{bedos2023c,
  title={{C}{*}-irreducibility for reduced twisted group {C}{*}-algebras},
  author={B{\'e}dos, E. and Omland, T.},
  journal={Journal of Functional Analysis},
  volume={284},
  number={5},
  pages={109795},
  year={2023},
  publisher={Elsevier}
}

@article{DM14,
  author = {Dudko, A. and Medynets, K.},
  title = {Finite factor representations of {H}igman--{T}hompson groups},
  journal = {Groups, Geometry, and Dynamics},
  volume = {8},
  number = {2},
  year = {2014},
  pages = {375--389}
}

@article{JL26a,
  author = {Jiang, Y. and Li, H.},
  title = {Classification of invariant subalgebras in a class of factors with property ({T})},
  journal = {arXiv preprint arXiv:2601.06353},
  year = {2026},
  pages={27},
}

@article{JL26b,
  author = {Jiang, Y. and Liu, R.},
  title = {On invariant subalgebras when the {ISR} property fails},
  journal = {J. Operator Theory},
  volume = {95},
  number = {1},
  year = {2026},
  pages = {103--117}
}

@article{JZ24,
  author = {Jiang, Y. and Zhou, X.},
  title = {An example of an infinite amenable group with the {ISR} property},
  journal = {Mathematische Zeitschrift},
  volume = {307},
  number = {2},
  year = {2024},
  pages = {23}
}

@article{KP23,
  author = {Kalantar, M. and Panagopoulos, N.},
  title = {On invariant subalgebras of group and von {N}eumann algebras},
  journal = {Ergodic Theory and Dynamical Systems},
  volume = {43},
  number = {10},
  year = {2023},
  pages = {3341--3353}
}

@book{Paulsen02,
  author = {Paulsen, V.},
  title = {Completely Bounded Maps and Operator Algebras},
  publisher = {Cambridge University Press},
  series = {Cambridge Studies in Advanced Mathematics},
  volume = {78},
  year = {2002}
}

@article{Popa07,
  author = {Popa, S.},
  title = {Deformation and rigidity for group actions and von {N}eumann algebras},
  journal = {International Congress of Mathematicians. Vol. I},
  publisher = {Eur. Math. Soc., Z\"urich},
  year = {2007},
  pages = {445--477}
}

@article{suzuki2020complete,
  title={Complete descriptions of intermediate operator algebras by intermediate extensions of dynamical systems},
  author={Suzuki, Yuhei},
  journal={Communications in Mathematical Physics},
  volume={375},
  number={2},
  pages={1273--1297},
  year={2020},
  publisher={Springer}
}

@article{AHO24,
  title={On the amenable subalgebras of group von Neumann algebras},
  author={Amrutam, T. and Hartman, Y. and Oppelmayer, H.},
  journal={Journal of Functional Analysis},
  volume={288},
  number={2},
  pages={110718},
  year={2025},
  publisher={Elsevier}
}

@article{jiang2026factors,
  title={Factors with prescribed number of invariant subalgebras not arising from subgroups},
  author={Jiang, Y. and Xu, Q.},
  journal={arXiv preprint arXiv:2605.01795},
  pages={27},
  year={2026}
}

@book{brown2008textrm,
  title={{C}{*}-Algebras and Finite-Dimensional Approximations},
  author={Brown, Nathanial Patrick and Ozawa, Narutaka},
  volume={88},
  year={2008},
  publisher={American Mathematical Soc.}
}

\end{document}